\documentclass[11pt,reqno]{article}

\usepackage[text={153mm,235mm},centering]{geometry}                
\usepackage{amssymb}
\usepackage{amsmath,amsthm}

\usepackage[all]{xy}
\usepackage{authblk}
\usepackage[hidelinks]{hyperref}

\newtheorem{theorem}{Theorem}[section]
\newtheorem{lemma}[theorem]{Lemma}

\newtheorem{corollary}[theorem]{Corollary}
\newtheorem{proposition}[theorem]{Proposition}
\newtheorem{question}[theorem]{Question}

\theoremstyle{definition}
\newtheorem{example}[theorem]{Example}
\newtheorem{definition}[theorem]{Definition}
\newtheorem*{ack}{Acknowledgements}

\newtheorem{remark}[theorem]{Remark}

\allowdisplaybreaks

\begin{document}

\title{Diagonalizable modular derivations on bigraded Poisson spaces
and Koszul duality}

\author[1]{Ruobing Chen\thanks{Email: astrocola@163.com}}

\author[2]{Di Zhang\thanks{Email: zhangdi5727@hdu.edu.cn}}

\renewcommand\Affilfont{\small}

\affil[1]{School of Mathematics, Sichuan University, Chengdu 610064, P.R. China}

\affil[2]{School of Mathematical Sciences, Hangzhou Dianzi University, Hangzhou 310018, China}

\date{}

\maketitle

\begin{abstract}
For two bigraded spaces $k^{m||n}$ and $k^{n\wedge m}$ equipped
with appropriate quadratic Poisson structures,
Koszul duality identifies the mixed Poisson differential
calculi and, in the unimodular case, their Poisson cohomologies as
Batalin--Vilkovisky algebras. In the current paper we prove that Koszul duality induces an isomorphism of Batalin--Vilkovisky
algebras between their Poisson cohomologies when the modular
derivation of one algebra is diagonalizable, including
non-unimodular cases.
\end{abstract}

\setcounter{tocdepth}{2}
{\small\linespread{1.0}\selectfont\tableofcontents}


\section{Introduction}\label{sect:intro}

For an ordinary superspace $k^{m|n}$, where $k$
is a field of characteristic zero, the even coordinates commute
with the odd coordinates.
The coordinate ring $\mathcal O(k^{m|n})$ is 
a Koszul algebra, 
 whereas the duals of an even
and an odd coordinate anticommute. Consequently, exchanging the
even and odd dimensions does not produce the required
dual function algebra; in other words, the Koszul
dual algebra of 
$\mathcal O(k^{m|n})$
is not $\mathcal O(k^{n|m})$. Following Chen and Yu \cite{CY},
they realize both algebras as functions on \emph{bigraded} spaces:
$$
\begin{aligned}
A&=\mathcal O(k^{m||n})
 =\operatorname{Sym}_k\{x_1,\ldots,x_m,y_1,\ldots,y_n\},\\
A^!&=\mathcal O(k^{n\wedge m})
 =\operatorname{Sym}_k\{\xi_1,\ldots,\xi_m,\eta_1,\ldots,\eta_n\}.
\end{aligned}
$$
Here the symmetric algebras are taken with respect to the bigraded sign rule, with coordinate degrees
$$
|x_i|=(0,0),\qquad |y_j|=(1,0),\qquad
|\xi_i|=(0,-1),\qquad |\eta_j|=(-1,-1).
$$
Writing $|a|=(p(a),w(a))$, with parity degree $p(a)$
and weight $w(a)$, multiplication of homogeneous elements satisfies
$ab=(-1)^{p(a)p(b)+w(a)w(b)}ba$.
Thus the $x_i$ are polynomial variables and the $y_j$ are
exterior variables, whereas the $\xi_i$ are exterior
variables and the $\eta_j$ are polynomial variables.
On the dual side, $\xi_i\eta_j=-\eta_j\xi_i$.
The second grading therefore records information essential
to the commutation relations. Section~\ref{subsect:bigraded-pair} lists the commutation
relations for both algebras.

In \cite[Theorem 2.6]{CY}, the first author together with Yu
identifies $A$ and $A^!$ as a
Koszul-dual pair. For corresponding quadratic Poisson
bivectors $\pi$ and $\pi^!$ of degree $(0,-2)$, we then
applied this result to study Poisson structures
on these two spaces. If the two Poisson structures
are dual to each other, then their Poisson
cohomology and \emph{mixed} Poisson homology are isomorphic 
as \emph{differential
calculi};
see \cite[Theorem 5.5]{CY}. More precisely, the isomorphisms
of Poisson cohomologies and homologies from 
$\mathcal O(k^{m||n})$ and $\mathcal O(k^{n\wedge m})$
 simultaneously preserve
the wedge product and Schouten bracket on cohomology,
the contraction action of cohomology on homology, and the mixed
differential on homology.

In \emph{loc. cit.}, we also studied the Poincar\'e duality
for the above two spaces.
For these bigraded spaces,
a closed mixed volume form which is a Poisson cycle
then identifies Poisson cohomology with mixed Poisson
homology by contraction. In this case, such a Poisson
structure is called \emph{unimodular}.
Transporting the mixed differential on the latter through this duality
gives the so-called Batalin-Vilkovisky (BV) operator.
That paper then proved that the standard
mixed volume forms on the above two bigraded spaces
correspond under Koszul duality,
that unimodularity is preserved, and that the resulting
BV algebras are isomorphic; see
\cite[Proposition~5.6 and Theorem~5.9]{CY}.

These results lead to the following natural question:

\begin{question}
Does the bigraded BV algebra structure persist when the
Poisson structure is non-unimodular?
\end{question}

The answer is affirmative when the modular derivation acting on the function algebra is diagonalizable over the ground field. Under this hypothesis, Poisson cohomology carries a BV algebra structure. For an admissible quadratic Koszul dual pair, the Koszul correspondence identifies the resulting BV algebras.

To prove these statements, we combine twisted Poincar\'e duality
with a zero-eigenvalue construction. Twisted duality was studied
by Launois and Richard for a class of quadratic Poisson algebras
\cite{launois2007}, and by Luo, Wang and Wu for polynomial
Poisson algebras \cite{luo2015}. 
With our conventions, a closed mixed volume form $\gamma$ determines the modular vector field $\upsilon=-\operatorname{Div}_\gamma(\pi)$ and the twisted Poisson boundary $\partial_\upsilon=\partial_\pi-\tilde\iota_\upsilon$.
Contraction with $\gamma$ identifies Poisson cohomology
with twisted Poisson homology. The identity
$\partial_\upsilon\tilde d+\tilde d\partial_\upsilon=-L_\upsilon$
shows that $\tilde d$ anticommutes with $\partial_\upsilon$
on the zero-eigenvalue subcomplex of $L_\upsilon$.
Under diagonalizability, this subcomplex computes the full twisted Poisson homology. Its mixed differential therefore defines an operator $B$ on that homology through the inclusion-induced isomorphism.

Our construction of the BV operator does not require the
Poisson structure to be quadratic. It applies to the
function algebra of any finite-dimensional bigraded space
with a Poisson bivector of degree $(0,-2)$ and a
diagonalizable modular derivation.

\begin{theorem}\label{mainthm1.1}
Let $\Sigma$ be a finite-dimensional bigraded vector
space over $k$, and let
$A=\mathcal{O}(\Sigma)$ be equipped with a Poisson
bivector $\pi$ of degree $(0,-2)$.
Let $\gamma$ be the standard mixed volume form on $A$.
If the modular derivation determined by $\gamma$ is diagonalizable over $k$,
then $\mathrm{HP}^{\bullet,\bullet}(A)$ admits a
canonical Batalin--Vilkovisky algebra structure whose operator
has cohomological degree $(0,-1)$ and generates
the induced Schouten--Nijenhuis bracket. 
\end{theorem}

The proof combines the zero-eigenvalue homotopy equivalence of Theorem~\ref{diagonalizable} with twisted Poincar\'e duality. Corollary~\ref{diffcaldual} supplies a differential calculus with duality, from which the BV operator is obtained as $\Delta=-\mathrm{PD}^{-1}\circ B\circ\mathrm{PD}$.

For the quadratic Koszul dual pair, Proposition~\ref{modulardual} identifies the modular vector fields and shows that diagonalizability is preserved. The correspondence intertwines the twisted Poisson boundaries and the zero-eigenvalue constructions, and sends the chosen mixed volume form of $A$ to that of $A^!$ (Theorem~\ref{isodiff}). These compatibilities give the following extension of the unimodular BV isomorphism.

\begin{theorem}\label{mainthm}
Let $A=(\mathcal{O}(k^{m||n}),\pi)$ be an
admissible quadratic Poisson algebra,
and let
$A^!=(\mathcal{O}(k^{n\wedge m}),\pi^!)$
be its Koszul dual Poisson algebra.
Equip $A$ and $A^!$ with the mixed volume forms
$$
\begin{aligned}
\gamma=dx_1\wedge\cdots\wedge dx_m
 \wedge y_1^*\wedge\cdots\wedge y_n^*,\quad
\gamma^!=\xi_1^*\wedge\cdots\wedge\xi_m^*
 \wedge d\eta_1\wedge\cdots\wedge d\eta_n.
\end{aligned}
$$
Assume that the modular derivation of $A$ determined by $\gamma$ is diagonalizable over $k$.
Then the Koszul correspondence induces an isomorphism
$$
    \mathrm{HP}^{\bullet,\bullet}(A)
    \xrightarrow{\;\sim\;}
    \mathrm{HP}^{\bullet,\bullet}(A^!)
$$
of Batalin--Vilkovisky algebras, with the BV operators
determined by $\gamma$ and $\gamma^!$.
\end{theorem}

When the modular vector field vanishes, the twisted Poisson
boundary reduces to the usual one, and the zero-eigenvalue
subcomplex becomes the whole complex. Our construction then
recovers the unimodular case of \cite{CY}.

The rest of the paper is organized as follows.
Section~\ref{sect:Poisson} introduces the bigraded pair,
proves twisted duality and constructs the homotopy
equivalences with the zero-eigenvalue subcomplexes. Section~\ref{sect:BV} obtains the
differential calculus with duality and proves
Theorem~\ref{mainthm1.1}. Section~\ref{sect:Koszul}
establishes compatibility with Koszul duality and proves
Theorem~\ref{mainthm}. Section~\ref{sect:examples}
works out some non-unimodular examples.

\section{Bigraded Poisson complexes and the modular obstruction}\label{sect:Poisson}

We recall the conventions for bigraded Poisson structures and mixed differential forms from \cite{CY}, and then establish twisted duality and the zero-eigenvalue contraction. Throughout this section, $A=\mathcal O(\Sigma)$ for a finite-dimensional bigraded space $\Sigma$. Section~\ref{sect:Koszul} applies these constructions to the quadratic Koszul dual pair introduced above.

\subsection{The bigraded Koszul pair and grading conventions}
\label{subsect:bigraded-pair}

A \textit{bigraded vector space} is a vector space equipped with a decomposition $$V=\bigoplus_{(i,j)\in \mathbb{Z}\times\mathbb{Z}}V_{i,j}.$$
An element $a\in V_{i,j}$ is homogeneous of degree $(i,j)$. We call $i$ its \textit{parity degree} and $j$ its \textit{weight}, and write $p(a)=i$ and $w(a)=j$.
A \textit{bigraded associative algebra} is a bigraded vector space $A$ equipped with an associative product which maps $A_{i_1,j_1} \times A_{i_2,j_2}$ to $A_{i_1+i_2, j_1+j_2}$.
The algebra $A$ is called bigraded commutative if
$$xy=(-1)^{p(x)p(y)+w(x)w(y)}yx,$$
for any homogeneous elements $x,y\in A$.

Throughout, $k$ is a field of characteristic zero.
For a homogeneous element $a$, write
$|a|=(p(a),w(a))$. In sign exponents, we use the pairing
$$
|a||b|:=p(a)p(b)+w(a)w(b).
$$
This is the Deligne sign convention \cite{deligne1999notes}. We recall the two bigraded spaces used below; see \cite{CY} for details.

\begin{example}
Let $k^{m||n}=\mathrm{Span}_k\{x_1,\cdots, x_m, y_1,\cdots, y_n\}$
with
$|x_i|=(0,0)$ and $|y_j|=(1,0)$.
The coordinate degrees thus give the relations
$$
x_i x_a=x_a x_i,\qquad
x_i y_j=y_j x_i,\qquad
y_j y_b=-y_b y_j.
$$
Thus, as an ordinary associative algebra,
$\mathcal O(k^{m||n})=k[x_1,\ldots,x_m]\otimes
\Lambda(y_1,\ldots,y_n)$.

For the second space, let $k^{n\wedge m}=\mathrm{Span}_k\{\xi_1,\cdots,\xi_m,\eta_1,\cdots,\eta_n\}$, with $|\xi_i|=(0,-1)$ and $|\eta_j|=(-1,-1)$.
The commutation relations are
$$
\xi_i\xi_a=-\xi_a\xi_i,\qquad
\xi_i\eta_j=-\eta_j\xi_i,\qquad
\eta_j\eta_b=\eta_b\eta_j.
$$
With these degrees, both function algebras are bigraded commutative and form a Koszul dual pair \cite[Theorem~2.6]{CY}.
\end{example}

Here \emph{polynomial degree} means total degree in the coordinate generators, each of which has degree one. It is independent of the parity degree and weight. We use polynomial degree to define quadratic Poisson structures and the bigrading to determine signs.

\subsection{Bigraded Poisson algebras and mixed complexes}

We recall the complexes and operations used to construct the BV operator for a finite-dimensional bigraded space $\Sigma$. Section~\ref{sect:Koszul} compares these constructions for the Koszul dual pair introduced above.

\begin{definition}[Bigraded Poisson algebra]
    Let $A$ be a bigraded commutative algebra. A bivector field $\pi$ satisfying $[\pi,\pi]=0$ defines a \textit{Poisson algebra} $(A,\pi)$.
\end{definition}
Here $[-,-]$ denotes the bigraded Schouten--Nijenhuis bracket
of degree $(0,1)$ in \cite[Definition~3.7]{CY}.
For homogeneous polyvectors it satisfies
$$
[P,Q]=-(-1)^{p(P)p(Q)+(w(P)+1)(w(Q)+1)}[Q,P].
$$
In particular, if $|\pi|=(0,-2)$, then
$[P,\pi]=(-1)^{w(P)}[\pi,P]$. A Poisson bivector $\pi$ induces the \textit{Poisson bracket}
\begin{align*}
    \{f,g\}:=\pi(df,dg),\ \mathrm{for\ any}\ f,g\in A, 
\end{align*}
on $A$. This bracket satisfies bigraded skew-commutativity, the Jacobi identity, and the Leibniz rule.

\begin{definition}[Derivation]
Let $A$ be a bigraded algebra. A \textit{derivation} of $A$ with values in an $A$-module $M$ is a $k$-linear map $\upsilon: A \to M$ satisfying
$$
\upsilon(ab) = \upsilon(a)b + (-1)^{p(a)p(\upsilon)+w(a)w(\upsilon)}a\upsilon(b) \quad \text{for all homogeneous } a,b \in A,
$$
and $\upsilon(c)=0$ for every $c \in k$. 
We denote the set of such derivations by $\overline{\mathfrak X}(A,M)$, abbreviated to $\overline{\mathfrak X}(A)$ or $\operatorname{Der}(A)$ when $M=A$.
\end{definition}

The $A$-module $\mathfrak X^1(A)$ has the same underlying
vector space as $\overline{\mathfrak X}(A)$, with the weight of
each homogeneous element decreased by one.

\begin{definition}[Poisson derivation]
    Suppose $A$ is a bigraded Poisson algebra. A \textit{Poisson derivation} of $A$ is a derivation $\upsilon \in \mathrm{Der}(A)$ satisfying $$\upsilon (\{f,g\})=\{\upsilon (f),g\}+(-1)^{p(f)p(\upsilon)+w(f)w(\upsilon)}\{f,\upsilon (g)\}$$
    for any homogeneous elements $f,g\in A$.
    
\end{definition}

\begin{definition}[Polyvectors]
Suppose $A$ is a bigraded commutative algebra.
For $p\ge 0$, the $A$-module of
$p$-th polyvectors, denoted by 
$\mathfrak{X}^p(A)$, is defined as the 
$p$-th bigraded symmetric product of
$
\mathfrak{X}^1(A)
$; that is,
$$
\mathfrak{X}^p(A):=\mathrm{Sym}^p_A(\mathfrak{X}^1(A)).
$$
The $A$-module of all polyvectors on $A$ is then
$$\mathfrak{X}^\bullet(A):=
\bigoplus_{p\ge 0}\mathfrak{X}^p(A).$$
\end{definition}

\begin{definition}[Poisson cochain complex]
Following \cite[Definition~4.3]{CY}, let
$A=\mathcal{O}(\Sigma)$ be a bigraded Poisson algebra with
$\deg(\pi)=(0,-2)$.
Its Poisson cochain complex is defined by
$$
\mathrm{CP}^{s,t}(A)
=
\{P\in\mathfrak{X}^{\bullet}(A)
\mid \deg(P)=(-s,-t)\},
\qquad
\delta^\pi=[-,\pi]:
\mathrm{CP}^{s,t}(A)\longrightarrow
\mathrm{CP}^{s,t+1}(A).
$$
Its cohomology is denoted by $\mathrm{HP}^{s,t}(A)$.
\end{definition}

We write $\mathfrak X^{a,b}(A)$ and
$\widetilde\Omega^{a,b}(A)$ for the components of degree $(a,b)$. Thus $\mathrm{CP}^{s,t}(A)=
\mathfrak X^{-s,-t}(A)$, whereas
$\widetilde{\mathrm{CP}}_{s,t}(A)=\widetilde\Omega^{s,t}(A)$.
A single superscript in $\mathfrak X^q(A)$ or
$\widetilde\Omega^q(A)$ instead denotes polyvector or mixed form
length, respectively.

For the Poisson chain complex, we recall the mixed differential forms of \cite{CY}. Their coefficient algebra is generated by the coordinates of even total parity and the derivatives with respect to coordinates of odd total parity.

Let $\Sigma$ be a bigraded vector space, and write
$A=\mathcal{O}(\Sigma)=\operatorname{Sym}
\langle z_1,\ldots,z_l\rangle$.
The homogeneous generators are ordered so that
$p(z_i)+w(z_i)$ is even for $1\leq i\leq m$ and odd
for $m<i\leq l$, where $0\leq m\leq l$.

Define the bigraded subalgebra of $\mathfrak X^\bullet(A)$ by
$$
\tilde{\Omega}^{0}(A):= \operatorname{Sym}\{ z_1,\dots ,z_m,\;
\partial_{m+1},\dots ,\partial_l \}.
$$
Define $\tilde d$ on the generators by $\tilde d(z_i)=dz_i$ for $i\leq m$ and $\tilde d(\partial_j)=z_j^*$ for $j\geq m+1$. Extend these assignments to a derivation of degree $(0,1)$,
$$
\tilde d:\tilde\Omega^0(A)\longrightarrow\tilde\Omega^1(A):=\Big\{\sum_{i=1}^m f_i dz_i+\sum_{j=m+1}^l f_j z_j^*\mid f_i,f_j\in\tilde\Omega^0(A)\Big\}.
$$
    
\begin{definition}[Mixed differential forms]
 For $p>0$, the space of \textit{mixed differential $p$-forms} is the $p$th bigraded symmetric power
    \begin{align*}
        \tilde{\Omega}^p(A)=\mathrm{Sym}^p_{\tilde{\Omega}^0(A)}(\tilde{\Omega}^1(A)).
    \end{align*}
The \textit{total space of mixed differential forms} is
    \begin{align*}
        \tilde{\Omega}^\bullet(A):=\bigoplus_{p\ge 0}\tilde{\Omega}^p(A)
        =\mathrm{Sym}\{z_1,\cdots,z_m,\partial_{m+1},\cdots,\partial_l,dz_1,
        \cdots,dz_m,z_{m+1}^*,\cdots,z_l^*\}.
    \end{align*}
\end{definition}
Extend $\tilde d$ to $\tilde\Omega^\bullet(A)$ by setting $\tilde d(dz_i)=\tilde d(z_j^*)=0$ and using the bigraded Leibniz rule. Its square vanishes on the generators and hence on the whole algebra. Thus $\tilde d$ is a differential of degree $(0,1)$ satisfying
\begin{align*}
    \tilde{d}(\omega\wedge \eta)=\tilde{d}\omega\wedge\eta+
    (-1)^{|\omega||\tilde{d}|}\omega\wedge \tilde{d}\eta,
\end{align*}
for any $\omega, \eta\in \tilde{\Omega}^{\bullet}(A)$.
This map $\tilde{d}$ is called the \textit{mixed differential}.

The generator pairings used for mixed contraction are
$$
\langle dz_i,\partial_j\rangle=\delta_{ij}
\quad(1\leq i,j\leq m),\qquad
\langle z_s^*,z_t\rangle=\delta_{st}
\quad(m<s,t\leq l).
$$
All other pairings of generators vanish.

\begin{definition}[Mixed vector fields]
The space of \textit{mixed vector fields} on $A$ is the subspace $$\tilde{\mathfrak{X}}^1(A):
=\Big\{\sum_{i=1}^{m}g_i\partial_{i}+\sum_{j=m+1}^{l}g_jz_j\mid g_i,g_j\in 
\tilde{\Omega}^0(A)\Big\}\subset\mathfrak X^\bullet(A).$$
For each integer $p\ge 1$, the space of \textit{mixed $p$-vector fields} is
$$
\tilde{\mathfrak X}^p(A):=\operatorname{Sym}_{\tilde\Omega^0(A)}^p(\tilde{\mathfrak X}^1(A))\subset\mathfrak X^\bullet(A),
$$
where the symmetric power is bigraded.
\end{definition}

\begin{definition}[Mixed contraction]
For a homogeneous mixed vector field $X$, the \textit{mixed contraction} $\tilde\iota_X$ is the derivation of bidegree $|X|$ on $\widetilde\Omega^\bullet(A)$ determined by $\tilde\iota_X(f)=0$ for $f\in\widetilde\Omega^0(A)$ and $\tilde\iota_X(\omega)=\langle\omega,X\rangle$ for $\omega\in\widetilde\Omega^1(A)$. It lowers mixed form length by one.
\end{definition}
The mixed contraction satisfies the bigraded Leibniz rule: for homogeneous elements $\alpha,\beta\in \tilde{\Omega}^\bullet(A)$,
\begin{align}\label{mixedcontraction:Leibniz}
    \tilde{\iota}_X(\alpha\wedge\beta)=(\tilde{\iota}_X\alpha)\wedge\beta+
    (-1)^{p(\alpha)p(X)+w(\alpha)w(X)}\alpha\wedge(\tilde{\iota}_X\beta).
\end{align}
For $f\in\widetilde\Omega^0(A)$, set $\tilde\iota_f(\omega)=f\omega$ and $\tilde\iota_{fP}=f\tilde\iota_P$. These operations extend to all
of $\mathfrak X^\bullet(A)$, viewed as the mixed polyvector
algebra over $\widetilde\Omega^0(A)$, by
$$
\tilde{\iota}_{P\wedge Q}:=\tilde{\iota}_P\circ\tilde{\iota}_Q \quad P,Q\in\mathfrak{X}^\bullet(A).
$$

\begin{proposition}[Mixed Cartan formula]\cite{CY}\label{mixedcartan}
Let $\Sigma$ be a bigraded space and let $A = \mathcal{O}(\Sigma)$ be defined as above. For any homogeneous mixed polyvector fields $P$ and $Q$, the following identity holds:
\begin{align}\label{mixedcartanformula}
    [[\tilde{\iota}_P,\tilde{d}],\tilde{\iota}_Q] = \tilde{\iota}_{[P,Q]}.
\end{align}
    \end{proposition}

\begin{definition}[Mixed Poisson chain complex]
Following \cite[Definition~4.5]{CY}, let
$A=\mathcal{O}(\Sigma)$ be a bigraded Poisson algebra with
$\deg(\pi)=(0,-2)$.
Its mixed Poisson chain complex is defined by
$$
\widetilde{\mathrm{CP}}_{s,t}(A)
=
\widetilde{\Omega}^{s,t}(A),
\qquad
\partial_\pi=[\tilde{\iota}_\pi,\tilde d]:
\widetilde{\mathrm{CP}}_{s,t}(A)\longrightarrow
\widetilde{\mathrm{CP}}_{s,t-1}(A).
$$
Its homology is denoted by
$\widetilde{\mathrm{HP}}_{s,t}(A)$.
\end{definition}

By \cite[Proposition~4.4]{CY},
$\partial_\pi^2=0$ and
$\partial_\pi\tilde d+\tilde d\partial_\pi=0$.

\subsection{The modular obstruction and twisted duality}

Contraction with a mixed volume form identifies the polyvector and mixed form spaces. The modular vector field measures the failure of this contraction map to intertwine the untwisted Poisson differentials. For the ordinary modular obstruction and unimodular duality, see \cite{wei1997,xu1999gerstenhaber}. We use the mixed volume formulation of \cite[Section~4.3]{CY} to establish twisted duality, and then use the eigenvalue construction to obtain the BV operator.

\begin{definition}[Diagonalizable derivation]
A derivation $\upsilon\in \text{Der}(A)$ of an algebra $ A $ over $ k $ is called \textit{diagonalizable} (or \textit{semisimple}) if $A$ admits a decomposition of $ k $-vector spaces
$$
A = \bigoplus_{\lambda \in S} A_\lambda, \quad A_\lambda := \{a \in A \mid \upsilon (a) = \lambda a\},
$$ for some subset $ S \subseteq k $.
\end{definition}

Since $\upsilon(1)=0$, we have $1\in A_0$. Thus $0$ is an eigenvalue of $\upsilon$ on $A$. A top form $\gamma\in\widetilde\Omega^{i,j}(A)$ is called a \textit{mixed volume form} of degree $(i,j)$ if the contraction map
$$
\tilde \iota_{(-)}\gamma:
\mathfrak{X}^{a,b}(A)
\longrightarrow
\widetilde{\Omega}^{i+a,j+b}(A),
\qquad P\longmapsto\tilde\iota_P\gamma
$$
is an isomorphism.
In cochain notation, this map has source $\mathrm{CP}^{s,t}(A)$ and target $\widetilde\Omega^{i-s,j-t}(A)$.

The bigraded sign rule makes $\widetilde\Omega^0(A)$ a
polynomial algebra. The ordered products of distinct mixed
one-form generators form a basis of
$\widetilde\Omega^\bullet(A)$ over $\widetilde\Omega^0(A)$.
Contraction with the ordered product of all mixed one-form generators identifies the mixed polyvector basis with the mixed form basis, up to signs.
The only units of $\widetilde\Omega^0(A)$ are nonzero
scalars, and $\tilde d$ annihilates the one-form generators.
These observations give the following proposition.
\begin{proposition}\label{prop:canonical-volume}
For every finite-dimensional bigraded space $\Sigma$, the form
$$
\gamma_{\mathrm{std}}=dz_1\wedge\cdots\wedge dz_m
\wedge z_{m+1}^*\wedge\cdots\wedge z_l^*
$$
is a mixed volume form. Every mixed volume form is a
nonzero scalar multiple of it, and is $\tilde d$-closed.
\end{proposition}

Fix the standard mixed volume form $\gamma$ from now on.
Contraction $\tilde\iota_{(-)}\gamma$ gives the following diagram
\begin{equation}\label{diag:unimodularPoisson}
\xymatrixcolsep{4pc}
\xymatrix{
\mathrm{CP}^{{\bullet,\bullet}} (A) \ar[r]^-{\tilde\iota_{(-)}\gamma} 
& \widetilde{\mathrm{CP}}_{i-\bullet,j-\bullet}(A) \\
\mathrm{CP}^{\bullet,\bullet-1} (A) \ar[r]^-{\tilde\iota_{(-)}\gamma} 
\ar[u]_{\delta^\pi} & \widetilde{\mathrm{CP}}_{i-\bullet,j-\bullet+1}(A).
\ar[u]_{\partial_\pi}
}
\end{equation}
We do not assume that the diagram in
\eqref{diag:unimodularPoisson} commutes.
The divergence operator is defined by the following commutative diagram:
\begin{equation}\label{diag:divergence}
\xymatrixcolsep{4pc}
\xymatrix{
\mathrm{CP}^{s,t}(A)
\ar[r]^-{D_\gamma}
\ar[d]_{\mathrm{Div}_\gamma}
&
\widetilde{\mathrm{CP}}_{i-s,j-t}(A)
\ar[d]^{\tilde d}
\\
\mathrm{CP}^{s,t-1}(A)
\ar[r]^-{D_\gamma}
&
\widetilde{\mathrm{CP}}_{i-s,j-t+1}(A)
}
\end{equation}
Here $D_\gamma(P)=\tilde{\iota}_P\gamma$, and
$\mathrm{Div}_\gamma:=D_\gamma^{-1}\circ\tilde d\circ D_\gamma$
is the divergence operator associated with $\gamma$.
In the coordinate description, every nonzero term of
$D_\gamma^{-1}\circ\tilde d\circ D_\gamma$ removes one coordinate
derivative from a polyvector monomial. Hence
$\mathrm{Div}_\gamma(\mathfrak X^r(A))
\subseteq\mathfrak X^{r-1}(A)$.
The modular vector field $\upsilon=-\mathrm{Div}_\gamma(\pi)$ belongs
to $\mathfrak X^1(A)$ with degree $(0,-1)$.
The corresponding derivation $\overline\upsilon$ of $A$
has degree $(0,0)$.

\begin{lemma}\label{lem:modular-poisson}
Let $\gamma$ be a closed mixed volume form and let
$\upsilon=-\mathrm{Div}_\gamma(\pi)$.
Then $[\pi,\upsilon]=0$.
\end{lemma}
\begin{proof}
Since $\tilde d\gamma=0$,
$\partial_\pi\gamma=-\tilde dD_\gamma(\pi)
=\tilde\iota_\upsilon\gamma$.
The operators $\partial_\pi$ and $\tilde\iota_\upsilon$
both have degree $(0,-1)$, so the mixed Cartan formula gives
$$
D_\gamma([\pi,\upsilon])
=[\partial_\pi,\tilde\iota_\upsilon]\gamma
=\partial_\pi^2\gamma+\tilde\iota_\upsilon^2\gamma=0.
$$
Here $\tilde\iota_\upsilon^2=\tilde\iota_{\upsilon\wedge\upsilon}=0$,
since $|\upsilon|=(0,-1)$ and $\operatorname{char}k=0$.
Injectivity of $D_\gamma$ proves the assertion.
\end{proof}

For a Poisson vector field $\upsilon\in\mathfrak X^1(A)$
of degree $(0,-1)$, define the \textit{twisted
mixed Poisson complex} directly by
$$
\widetilde{\mathrm{CP}}_{s,t}(A;\upsilon)
:=\widetilde\Omega^{s,t}(A),\qquad
\partial_\upsilon:=\partial_\pi-\tilde\iota_\upsilon.
$$
The mixed Cartan identity gives
$[\partial_\pi,\tilde\iota_\upsilon]=
\tilde\iota_{[\pi,\upsilon]}=0$, and therefore
$\partial_\upsilon^2=0$. Its homology is denoted by
$\widetilde{\mathrm{HP}}_{s,t}(A;\upsilon)$.
For the modular vector field, the twist is precisely the
correction required by volume contraction in
Proposition~\ref{prop2.14}.

Following \cite{CY}, unimodularity means the existence of a
mixed volume form which is a Poisson cycle.
By Proposition~\ref{prop:canonical-volume} and
$\partial_\pi\gamma=\tilde\iota_\upsilon\gamma$, this is
equivalent to $\upsilon=0$. The twisted boundary is
chosen so that $\partial_\upsilon\gamma=0$ even when
$\upsilon\ne0$. The next proposition extends this identity
from the volume form to contraction by arbitrary polyvectors.

\begin{proposition}\label{prop2.14}
Let $A=\mathcal O(\Sigma)$ have a Poisson bivector $\pi$ of degree $(0,-2)$ and a mixed volume form $\gamma$, which is $\tilde d$-closed by Proposition~\ref{prop:canonical-volume}. Then every homogeneous polyvector $P$ satisfies
\begin{align}\label{twistedpartial}
(-1)^{w(P)}\tilde{\iota}_{\delta^\pi(P)}\gamma
=
\partial_\pi(\tilde{\iota}_P\gamma)
-
\tilde{\iota}_\upsilon(\tilde{\iota}_P\gamma),
\end{align}
where $\upsilon$ denotes the modular vector field associated with $\gamma$.
\end{proposition}
\begin{proof}
    By the mixed Cartan formula \eqref{mixedcartanformula}, 
    \begin{align}\label{cartan}
\tilde{\iota}_{[P,\pi]}(\gamma)
=
(-1)^{w(P)}
\partial_\pi(\tilde{\iota}_P\gamma)
-
\tilde{\iota}_P(\partial_\pi\gamma).
\end{align}
Since $\tilde d\gamma=0$, we have $\partial_\pi\gamma=\tilde\iota_\upsilon\gamma$. Bigraded commutativity of contractions then gives
$$
\tilde\iota_\upsilon\tilde\iota_P\gamma
=(-1)^{|P||\upsilon|}\tilde\iota_P\tilde\iota_\upsilon\gamma
=(-1)^{w(P)}\tilde\iota_P(\partial_\pi\gamma).
$$
    The result follows by substitution into \eqref{cartan}.
\end{proof}

\subsection{The zero-eigenvalue construction}
The twisted complex provides the homology required for duality. However, $\tilde d$ need not anticommute with $\partial_\upsilon$ and therefore may not induce an operator on this homology. Following \cite[Section~3]{chen2023batalin}, we identify their anticommutator with the modular action and construct a contracting homotopy on each nonzero eigenspace.

Luo, Wang and Wu construct BV structures for smooth Poisson algebras with trivial canonical bundle under the following condition: the modular vector field is the contraction of $\pi$ with a closed differential one-form $\varpi$. Their construction modifies the de Rham differential \cite{luo2024}.
For ordinary polynomial Poisson algebras, every such form
is exact, so their condition is equivalent to the modular
vector field being Hamiltonian \cite{luo2024}.

Our operator is defined using the zero-eigenvalue subcomplex of the modular action. Diagonalizability does not require the modular vector field to be Hamiltonian: in Example~\ref{ex:nondiagonal}, the modular vector field is nonzero and linear, whereas a Hamiltonian vector field for a quadratic Poisson structure has no nonzero linear part. Thus this example does not belong to the polynomial class satisfying the condition of Luo, Wang and Wu.

Let $\Sigma$ be a finite-dimensional bigraded vector space
over a field $k$ of characteristic zero, and let
$A=\mathcal{O}(\Sigma)$ be equipped with a Poisson structure
$\pi$ of degree $(0,-2)$.
Let $\gamma\in\widetilde{\Omega}^{u,v}(A)$ be the
standard mixed volume form; it is closed by
Proposition~\ref{prop:canonical-volume}.
Let $\upsilon=-\mathrm{Div}_{\gamma}(\pi)$ be the associated
modular Poisson vector field.
Its degree as an element of $\mathfrak{X}^1(A)$ is
$(0,-1)$; the corresponding derivation of $A$, denoted by
$\overline{\upsilon}$, has degree $(0,0)$.
We use the twisted mixed Poisson boundary
$\partial_\upsilon=\partial_\pi-\tilde\iota_\upsilon$
and set $L_\upsilon=[\tilde\iota_\upsilon,\tilde d]$.

\begin{theorem}\label{horiso}
Contraction with $\gamma$ induces an isomorphism
\begin{equation}\label{twistedpoincare}
\mathrm{HP}^{\bullet,\bullet}(A)
\cong
\widetilde{\mathrm{HP}}_{u-\bullet,v-\bullet}
(A;\upsilon).
\end{equation}
\end{theorem}

\begin{proof}
Since $\gamma$ is a mixed volume form, the map
$$
D_\gamma:
\mathrm{CP}^{s,t}(A)
\longrightarrow
\widetilde{\mathrm{CP}}_{u-s,v-t}(A;\upsilon),
\qquad
D_\gamma(P)=\tilde\iota_P\gamma,
$$
is an isomorphism of vector spaces.
By Proposition~\ref{prop2.14},
$$
\partial_\upsilon D_\gamma(P)
=(-1)^tD_\gamma(\delta^\pi P),
\qquad P\in\mathrm{CP}^{s,t}(A),
$$
because $w(P)=-t$.
Thus $D_\gamma$ identifies cycles and boundaries
in the two complexes.

More precisely, define
$$
F(P)=(-1)^{t(t-1)/2}D_\gamma(P),
\qquad P\in\mathrm{CP}^{s,t}(A).
$$
Since $\delta^\pi P\in\mathrm{CP}^{s,t+1}(A)$, we obtain
$$
F(\delta^\pi P)
=(-1)^{t(t+1)/2}D_\gamma(\delta^\pi P)
=\partial_\upsilon F(P).
$$
Hence $F$ is an isomorphism of complexes after reindexing
the target by $(s,t)\mapsto(u-s,v-t)$.
The asserted isomorphism follows.
\end{proof}

\begin{lemma}
The operators $\partial_\upsilon$ and $\tilde d$ satisfy
\begin{equation}\label{homotopy}
\partial_\upsilon\tilde d+\tilde d\partial_\upsilon
=-L_\upsilon.
\end{equation}
Moreover, $L_\upsilon$ commutes with both
$\partial_\upsilon$ and $\tilde d$, and
$L_\upsilon\gamma=0$.
\end{lemma}

\begin{proof}
Using
$\partial_\pi\tilde d+\tilde d\partial_\pi=0$, we have
$$
\partial_\upsilon\tilde d+\tilde d\partial_\upsilon
=
-\bigl(\tilde\iota_\upsilon\tilde d
+\tilde d\tilde\iota_\upsilon\bigr)
=-L_\upsilon.
$$
Since $\partial_\upsilon^2=\tilde d^2=0$,
equation~\eqref{homotopy} implies
$[L_\upsilon,\partial_\upsilon]=[L_\upsilon,\tilde d]=0$.
Proposition~\ref{prop2.14}, applied to $P=1$, gives
$\partial_\upsilon\gamma=0$.
Together with $\tilde d\gamma=0$, this yields
$L_\upsilon\gamma=0$.
\end{proof}

\begin{definition}With the notation above, 
for $\lambda\in k$, set
\begin{align*}
\mathrm{CP}^{s,t}_\lambda(A)
&:=
\{P\in\mathrm{CP}^{s,t}(A)
\mid[\upsilon,P]=\lambda P\},\\
\widetilde{\mathrm{CP}}^\lambda_{s,t}(A;\upsilon)
&:=
\{\alpha\in\widetilde{\mathrm{CP}}_{s,t}(A;\upsilon)
\mid L_\upsilon\alpha=\lambda\alpha\}.
\end{align*}
\end{definition}

The identity $[\upsilon,\pi]=0$ implies that $\delta^\pi$
preserves the cochain eigenspaces
$\mathrm{CP}^{s,t}_\lambda(A)$. The preceding lemma shows that
$\partial_\upsilon$ and $\tilde d$ preserve the chain
eigenspaces $\widetilde{\mathrm{CP}}^\lambda_{s,t}(A;\upsilon)$.
Consequently, these eigenspaces define subcomplexes.

Since $L_\upsilon\gamma=0$, the mixed Cartan identity gives
$$
L_\upsilon D_\gamma(P)
=
L_\upsilon(\tilde\iota_P\gamma)
=
\tilde\iota_{[\upsilon,P]}\gamma
=
D_\gamma([\upsilon,P]).
$$
Consequently, $D_\gamma$ preserves eigenvalues, as does the isomorphism of complexes $F$ constructed in the proof of Theorem~\ref{horiso}.

\begin{theorem}\label{diagonalizable}
Suppose that $\overline{\upsilon}$ is diagonalizable
on $A$ over $k$.
Then the inclusions of the zero-eigenvalue subcomplexes
are homotopy equivalences. In particular,
$$
\widetilde{\mathrm{HP}}_{\bullet,\bullet}(A;\upsilon)
\cong
\mathrm{H}_{\bullet,\bullet}
\bigl(\widetilde{\mathrm{CP}}^0_{\bullet,\bullet}
(A;\upsilon)\bigr),\quad
\mathrm{HP}^{\bullet,\bullet}(A)
\cong
\mathrm{H}^{\bullet,\bullet}
\bigl(\mathrm{CP}^{\bullet,\bullet}_0(A)\bigr).
$$
\end{theorem}

\begin{proof}
To obtain the eigenspace decompositions, take the eigencomponents of a finite homogeneous generating set of $A$. Their span $W\subset A$ is a finite-dimensional homogeneous $\overline\upsilon$-invariant subspace, because $\overline\upsilon$ is diagonalizable and preserves the bigrading. The subspace $W$ still generates $A$.

Restriction to $W$ gives an injective equivariant map
$$
\operatorname{Der}_k(A)
\longrightarrow
\operatorname{Hom}_k(W,A),
\qquad
D\longmapsto D|_W,
$$
where the action on the target is
$T\mapsto\overline\upsilon\circ T
-T\circ\overline\upsilon|_W$.
Choose an eigenbasis $e_1,\ldots,e_N$ of $W$, with $\overline\upsilon(e_i)=\lambda_i e_i$. Evaluation on this basis identifies the action on $\operatorname{Hom}_k(W,A)$ with $\bigoplus_{i=1}^N(\overline\upsilon-\lambda_i)$ on $A^{\oplus N}$, which is diagonalizable.

Every invariant subspace of a diagonalizable operator
is a direct sum of eigenspaces: the finitely many
eigencomponents of any vector are recovered by
polynomial spectral projections. The same argument
applies to invariant quotients. Consequently,
$[\overline\upsilon,-]$ is diagonalizable on
$\operatorname{Der}_k(A)$, and the shifted action
$[\upsilon,-]$ is diagonalizable on $\mathfrak X^1(A)$.
Since this action is a derivation of the wedge product,
tensor products and invariant quotients give
diagonalizability on
$\mathfrak X^r(A)=\operatorname{Sym}_A^r\mathfrak X^1(A)$
for every $r\geq1$. Together with the hypothesis on
$\mathfrak X^0(A)=A$, this yields
$$
\mathrm{CP}^{s,t}(A)
=
\bigoplus_{\lambda\in k}\mathrm{CP}^{s,t}_\lambda(A).
$$
Transporting this decomposition along the
eigenvalue-preserving isomorphism $D_\gamma$, we obtain
$$
\widetilde{\mathrm{CP}}_{s,t}(A;\upsilon)
=
\bigoplus_{\lambda\in k}
\widetilde{\mathrm{CP}}^\lambda_{s,t}(A;\upsilon).
$$
Both decompositions are algebraic direct sums.

Let $\iota_0$ and $p_0$ denote the inclusion of the
zero-eigenvalue chain subcomplex and the projection
onto it. By \eqref{homotopy},
\begin{equation}\label{homotopyid}
\left.
(\partial_\upsilon\tilde d+\tilde d\partial_\upsilon)
\right|_{\widetilde{\mathrm{CP}}^\lambda_{\bullet,\bullet}
(A;\upsilon)}
=
-\lambda\,\mathrm{id}.
\end{equation}
Define an operator $h$ of homological degree $(0,1)$ by
$$
h\big|_{\widetilde{\mathrm{CP}}^\lambda_{\bullet,\bullet}
(A;\upsilon)}
=
\begin{cases}
-\lambda^{-1}\tilde d,&\lambda\neq0,\\
0,&\lambda=0.
\end{cases}
$$
This operator is well defined because $\tilde d$
preserves eigenspaces and every element has only finitely many nonzero eigencomponents.
Equation~\eqref{homotopyid} gives
$$
\partial_\upsilon h+h\partial_\upsilon
=
\mathrm{id}-\iota_0p_0,
\qquad
p_0\iota_0=\mathrm{id}.
$$
Thus $\iota_0$ is a homotopy equivalence.

For cochains, let $F$ denote the isomorphism of complexes
constructed in Theorem~\ref{horiso}, and let $i_0$ and $q_0$ denote
the zero-eigenvalue inclusion and projection.
Then $H=F^{-1}hF$ has cohomological degree $(0,-1)$
and satisfies
$$
\delta^\pi H+H\delta^\pi
=
F^{-1}(\mathrm{id}-\iota_0p_0)F
=
\mathrm{id}-i_0q_0,
\qquad
q_0i_0=\mathrm{id}.
$$
Hence $i_0$ is also a homotopy equivalence.
The asserted homology and cohomology isomorphisms follow.
\end{proof}

For the analogous eigenvalue-one construction in twisted Hochschild homology,
see also \cite{KK}.
Under the hypothesis of Theorem~\ref{diagonalizable},
equation~\eqref{homotopyid} shows that
$\bigl(
\widetilde{\mathrm{CP}}^0_{\bullet,\bullet}(A;\upsilon),
\partial_\upsilon,\tilde d
\bigr)$
is a mixed complex.
The operator $\tilde d$ on this subcomplex induces
a square-zero operator on the full twisted Poisson
homology by transport:
$$
B :=
(\iota_0)_*
\circ
\bigl(\tilde d|_{\widetilde{\mathrm{CP}}^0
_{\bullet,\bullet}(A;\upsilon)}\bigr)_*
\circ
(p_0)_*.
$$
It has homological degree $(0,1)$.
In what follows, the notation $\tilde d$ on
$\widetilde{\mathrm{HP}}_{\bullet,\bullet}(A;\upsilon)$
refers to this transported operator.

\section{Differential calculus and Batalin--Vilkovisky algebras}\label{sect:BV}
The preceding section provides twisted Poincar\'e duality
and a mixed differential on twisted Poisson homology.
The construction of a BV operator requires compatibility
of these maps with the wedge product, Schouten bracket,
and contraction.
We verify this compatibility on the zero-eigenvalue
subcomplexes, obtain a differential calculus with duality,
and transport the mixed differential to cohomology.

\subsection{The calculus on twisted Poisson homology}
We recall a bigraded formulation of the differential calculus of Tamarkin and Tsygan \cite[Definition~3.2.2]{tamarkin2005ring}.
\begin{definition}[Differential calculus]\label{diffcalculus}
Let $\mathrm{H}^{\bullet,\bullet}$ and $\mathrm{H}_{\bullet,\bullet}$ 
be bigraded vector spaces.
A \textit{differential calculus} is a sextuple
$$
(\mathrm{H}^{\bullet,\bullet},
\mathrm{H}_{\bullet,\bullet}, \cup, \iota, [-,-], d)
$$
satisfying the following conditions:
\begin{enumerate}
\item[(1)] $(\mathrm{H}^{\bullet,\bullet},\cup,[-,-])$ is a 
bigraded Gerstenhaber
algebra.
\begin{itemize}
\item $(\mathrm{H}^{\bullet,\bullet},\cup)$ is a bigraded commutative algebra;
\item $(\mathrm{H}^{\bullet,\bullet},[-,-])$ is a Lie algebra of cohomological degree $(0,-1)$,
corresponding to actual bracket degree $(0,1)$ on polyvectors;
\item the bracket satisfies the bigraded Leibniz rule
$$
[P,Q\cup R]
=[P,Q]\cup R
+(-1)^{|P||Q|+w(Q)}\,Q\cup[P,R],
$$
for homogeneous elements $P, Q, R$;
\end{itemize}

\item[(2)] $\mathrm{H}_{\bullet,\bullet}$ is a bigraded (left) 
module over 
$(\mathrm{H}^{\bullet,\bullet}, \cup)$ via the map
$$
\iota: \mathrm{H}^{i,j}\otimes\mathrm{H}_{s,t}\to 
\mathrm{H}_{s-i,t-j},\;f\otimes p\mapsto \iota_fp,
$$
for any $f\in \mathrm{H}^{i,j}$ and $p \in \mathrm{H}_{s,t}$;
                
\item[(3)] There is a map $d: \mathrm{H}_{\bullet,\bullet}\to \mathrm{H}_{\bullet,\bullet+1}$ 
satisfying $d^2=0$.
With $L_f:=[\iota_f,d]=\iota_fd-(-1)^{w(f)}d\iota_f$, the Cartan identity for homogeneous $f,g\in\mathrm H^{\bullet,\bullet}$ is
\begin{align}\label{calculusformula}
\iota_{[f,g]}=[L_f,\iota_g] :={L}_{f}\iota_g-(-1)^{p(g)p(f)
+w(g)(w(f)+1)}\iota_g {L}_f.
\end{align}       
\end{enumerate}
\end{definition}

\begin{example}
For a finite-dimensional bigraded space $\Sigma$ over $k$
and a Poisson bivector $\pi$ of degree $(0,-2)$ on
$A=\mathcal O(\Sigma)$, the sextuple
$$
(\mathrm{HP}^{\bullet,\bullet}(A),
\widetilde{\mathrm{HP}}_{\bullet,\bullet}(A),
\wedge,\tilde\iota,[-,-],\tilde d)
$$
is a differential calculus by \cite[Theorem~4.6]{CY}.
In particular, this applies to both members of the
admissible quadratic Koszul pair.
\end{example}

The following notion is a bigraded version of Lambre's calculus with duality \cite{lambre2010dualite}.
\begin{definition}[differential calculus with duality]
\label{diff calculus}
A differential calculus $(\mathrm{H}^{\bullet,\bullet},
\mathrm{H}_{\bullet,\bullet}, \cup, \iota, [-,-], d)$ is
called a \textit{differential calculus with duality} if 
there exist integers $u,v$ and an element $\gamma\in \mathrm{H}_{u,v}$
such that
\begin{enumerate}
\item[(1)]  For the unit $1\in\mathrm H^{0,0}$, we have $\iota_1\gamma=\gamma$ and $d(\gamma)=0$;
\item[(2)]  for any $i,j\in \mathbb Z$,
\begin{equation*}\label{PD}
\mathrm{PD}(-):=\iota_{(-)}\gamma : \mathrm{H}^{i,j}
\to \mathrm{H}_{u-i,v-j}
\end{equation*}
is an isomorphism.
\end{enumerate}
We call the isomorphism $\mathrm{PD}$ the
\textit{Van den Bergh duality} (or
\textit{noncommutative Poincar\'e duality}) between
$\mathrm{H}^{\bullet,\bullet}$ and
$\mathrm{H}_{\bullet,\bullet}$.
\end{definition}

We now construct the calculus with duality using twisted Poisson homology and the operator $B$. The zero-eigenvalue subcomplexes allow us to treat the calculus operations and volume class together, as required in the non-unimodular case.

\begin{corollary}\label{diffcaldual}
Let $\Sigma$ be a finite-dimensional bigraded vector space
over a field $k$ of characteristic zero, and let
$A=\mathcal{O}(\Sigma)$ be equipped with a Poisson structure
$\pi$ of degree $(0,-2)$.
Let $\gamma\in\widetilde{\Omega}^{u,v}(A)$ be the
standard mixed volume form.
Assume that the derivation of $A$ corresponding to the modular vector field $\upsilon$ is diagonalizable over $k$.
Then
$$
\bigl(
\mathrm{HP}^{\bullet,\bullet}(A),
\widetilde{\mathrm{HP}}_{\bullet,\bullet}(A;\upsilon),
\wedge,\tilde\iota,[-,-],\tilde d
\bigr)
$$
is a differential calculus with duality,
with fundamental class $[\gamma]$.
\end{corollary}

\begin{proof}
The mixed Cartan identity and the graded commutativity
of contractions give, for every homogeneous polyvector $P$,
$$
[\partial_\upsilon,\tilde\iota_P]
=
\tilde\iota_{[\pi,P]}
=
(-1)^{w(P)}\tilde\iota_{\delta^\pi P}.
$$
Hence contraction by a Poisson cocycle preserves twisted
cycles and boundaries, while contraction by a Poisson
coboundary induces the zero map on twisted homology.

Moreover,
$[L_\upsilon,\tilde\iota_P]
=
\tilde\iota_{[\upsilon,P]}$.
Thus contraction by a zero-eigenvalue cochain preserves
the zero-eigenvalue chain complex. Since
$[\upsilon,-]$ is a derivation of both the wedge product
and the Schouten bracket, these operations also preserve
the zero-eigenvalue cochain complex.

On the zero-eigenvalue chain complex,
$\partial_\upsilon\tilde d+\tilde d\partial_\upsilon=0$.
The mixed Cartan identity therefore descends to homology,
giving a differential calculus. By
Theorem~\ref{diagonalizable}, the zero-eigenvalue inclusions
identify this calculus with the asserted one, where
$\tilde d$ denotes the transported operator $B$.

Finally, $\partial_\upsilon\gamma=\tilde d\gamma=0$, and contraction with $[\gamma]$ gives the duality isomorphism of Theorem~\ref{horiso}.
\end{proof}

\subsection{The Batalin--Vilkovisky operator}\label{subsect:bv-operator}
The BV operator is obtained from $B$ by volume duality. We use \cite[Theorem~4.11]{CY} with the convention $\Delta(a)=(-1)^{w(a)}\Delta_{\mathrm{CY}}(a)$, where $\Delta_{\mathrm{CY}}(a)=(-1)^{w(a)+1}\mathrm{PD}^{-1}d\mathrm{PD}(a)$. We derive the resulting bracket identity below in the form of \cite[Proposition~4.9, equation~(23)]{CY}. Here, in the non-unimodular case, $B$ comes from the zero-eigenvalue construction on twisted Poisson homology.
For the classical construction of graded Lie brackets from
second-order differential operators, see
\cite[Section~1]{koszul1985}.

\begin{definition}[Batalin--Vilkovisky algebra]
A {\it Batalin--Vilkovisky
algebra} is a bigraded commutative algebra $A$ equipped with a linear operator $\Delta: A^{i,j}\to A^{i,j-1}$ of degree $(0,-1)$, such that
\begin{enumerate}
\item[$(1)$] $\Delta$ is 
a differential, i.e., $\Delta^2=0$; 
\item[$(2)$] $\Delta$ is a second-order operator; that is,
\begin{align*}
\Delta(a  b c)
=& \Delta(a b)c+(-1)^{|a||\Delta|}a 
\Delta(b  c)+(-1)^{|b|(|a|+|\Delta|)}b  
\Delta(a   c)\nonumber \\
&- (\Delta a)  b   c-(-1)^{|a||\Delta|}a 
(\Delta b) c-(-1)^{|\Delta|(|a|+|b|)}a  b(\Delta c).
\end{align*}
\end{enumerate}
\end{definition}

With the sign convention above, the bigraded version \cite[Theorem~4.11]{CY} of Lambre's construction \cite{lambre2010dualite} takes the following form.
\begin{theorem}[\cite{CY}]\label{thm:bv-duality}
Let
$(\mathrm H^{\bullet,\bullet},\mathrm H_{\bullet,\bullet},
\cup,\iota,[-,-],d,\gamma)$
be a differential calculus with duality, and write
$\mathrm{PD}(a)=\iota_a\gamma$. Then
$\Delta=-\mathrm{PD}^{-1}\circ d\circ\mathrm{PD}$
has degree $(0,-1)$ and defines a Batalin--Vilkovisky
algebra structure on $\mathrm H^{\bullet,\bullet}$. For homogeneous
$a,b\in \mathrm H^{\bullet,\bullet}$, it satisfies
$$
[a,b]
=(-1)^{w(a)}
\bigl(\Delta(a\cup b)-\Delta(a)\cup b\bigr)
-a\cup\Delta(b).
$$
\end{theorem}

\begin{proof}
The degree of $\Delta$ is $(0,-1)$, since $\mathrm{PD}:\mathrm H^{s,t}\to\mathrm H_{u-s,v-t}$ and $d:\mathrm H_{s,t}\to\mathrm H_{s,t+1}$.
Moreover, $d^2=0$ and $d\gamma=0$ imply
$\Delta^2=0$ and $\Delta(1)=0$.

Applying the Cartan identity
$\iota_{[a,b]}=[[\iota_a,d],\iota_b]$ to $\gamma$,
and using $d\gamma=0$,
$\iota_{a\cup b}=\iota_a\iota_b$ and
$\mathrm{PD}^{-1}(\iota_a\omega)
=a\cup\mathrm{PD}^{-1}(\omega)$, we obtain
\begin{align}\label{eq:second-order}
   [a,b]
=(-1)^{w(a)}
\bigl(\Delta(a\cup b)-\Delta(a)\cup b\bigr)
-a\cup\Delta(b). 
\end{align}
The Leibniz rule for $[a,-]$, applied to \eqref{eq:second-order}, gives the second-order identity for $\Delta$; compare \cite[Proposition~1.2]{getzler1994batalin}. Together with $\Delta^2=0$, this makes $\Delta$ a BV operator.
\end{proof}

For BV structures on the Hochschild cohomology of Frobenius algebras with semisimple Nakayama automorphism, see also \cite{lambre2016}.

\begin{proof}[Proof of Theorem~\ref{mainthm1.1}]
Apply Theorem~\ref{thm:bv-duality} to the calculus with
duality in Corollary~\ref{diffcaldual}, using the operator
$B$ obtained from the zero-eigenvalue complex.
By Proposition~\ref{prop:canonical-volume}, any other mixed volume form
is $c\gamma$ with $c\in k^\times$. This leaves the modular
vector field, its eigenspaces and $B$ unchanged, and
replaces $\mathrm{PD}$ by $c\mathrm{PD}$. Hence
$\Delta=-\mathrm{PD}^{-1}B\mathrm{PD}$ is unchanged.
\end{proof}

\begin{remark}\label{rem:modular-class-bv}
For a Poisson cocycle $P$ in the zero eigenspace,
$\Delta[P]=[-\mathrm{Div}_\gamma(P)]$.
In particular,
$
\Delta[\pi]=[\upsilon].
$
Indeed, $[\pi,\pi]=0$ and Lemma~\ref{lem:modular-poisson} show that $\pi$ is a Poisson cocycle in the zero eigenspace. Substituting $\upsilon=-\mathrm{Div}_\gamma(\pi)$ gives the formula.
\end{remark}

\section{Koszul duality beyond unimodularity}\label{sect:Koszul}

We return to the function algebras of $k^{m||n}$ and $k^{n\wedge m}$ from Section~\ref{subsect:bigraded-pair}, with the same coordinate degrees and sign convention. Starting from the calculus isomorphism of \cite[Theorem~5.5]{CY}, we compare the modular vector fields and prove that diagonalizability is preserved. We then show that the correspondence respects twisted duality and the zero-eigenvalue construction, and hence identifies the BV operators.

\subsection{Koszul duality of quadratic Poisson structures}

Exchanging variables and derivatives sends a quadratic bivector to a bivector on the Koszul dual space. We recall this correspondence and compare the associated modular vector fields. Diagonal Poisson structures provide examples with diagonalizable modular derivations, while the comparison theorem requires only diagonalizability. All symmetric algebras and polyvector products use the bigraded sign rule. For Koszul duality after quadratic deformation quantization, see \cite{shoikhet2010Koszul}.

\begin{definition}[Quadratic Poisson structure]
Let $\Sigma$ be a finite-dimensional bigraded vector space,
and let $A=\mathcal{O}(\Sigma)=\operatorname{Sym}_k(\Sigma^*)$.
A Poisson structure $\pi$ on $A$ is called \emph{quadratic}
if its coefficients are homogeneous polynomials of degree two.
Equivalently, in homogeneous linear coordinates
$z_1,\ldots,z_l$, it has the form
$$
\pi=\sum_{a,b,c,d=1}^{l}
C^{a,b}_{c,d}\,z_a z_b\,
\partial_{z_c}\wedge\partial_{z_d},
\qquad C^{a,b}_{c,d}\in k.
$$
Here polynomial degree refers to the symmetric-algebra grading,
independently of the bigrading. Following \cite{CY}, a
quadratic Poisson structure of degree $(0,-2)$ is
called \emph{admissible quadratic}.
\end{definition}

\begin{definition}[Diagonal Poisson structure]
\label{def:diagonalPoisson}
Let $\Sigma$ be a finite-dimensional bigraded vector space,
and let $A=\mathcal{O}(\Sigma)$.
A Poisson structure $\pi$ on $A$ is called \emph{diagonal}
if there exist homogeneous linear coordinates
$z_1,\ldots,z_l$ such that
$$
\pi=\sum_{1\leq i<j\leq l}
c_{ij}z_i z_j\,
\partial_{z_i}\wedge\partial_{z_j},
\qquad c_{ij}\in k.
$$
In this case, $\pi$ is said to be diagonal with respect
to these coordinates.
\end{definition}

Every diagonal Poisson structure is quadratic.
Under the shifted polyvector convention,
$|\partial_{z_i}|=-|z_i|+(0,-1)$, so each term in the
diagonal expression has degree $(0,-2)$.
Here diagonality is a condition on the Poisson bivector in suitable homogeneous linear coordinates, whereas diagonalizability is a condition on its modular derivation.

For the Koszul dual spaces $k^{m||n}$ and $k^{n\wedge m}$, consider the function algebras
$$
\mathcal{O}(k^{m||n}) = \operatorname{Sym}\{x_1,\dots,x_m, y_1,\dots,y_n \},
\qquad
\mathcal{O}(k^{n\wedge m}) = \operatorname{Sym}\{ \xi_1,\dots,\xi_m, \eta_1,\dots,\eta_n \}.
$$
Write $(z_1,\ldots,z_{m+n})
=(x_1,\ldots,x_m,y_1,\ldots,y_n)$ and
$(\zeta_1,\ldots,\zeta_{m+n})
=(\xi_1,\ldots,\xi_m,\eta_1,\ldots,\eta_n)$.
We define a correspondence between the polyvector fields on the two algebras by exchanging variables and derivatives as follows:
\begin{align}\label{corres}
    z_i\leftrightarrow 
{\partial_{\zeta_i}},\quad{\partial_{z_i}}\leftrightarrow
\zeta_i,\quad i=1,\cdots,m+n.
\end{align}
Set $\epsilon_i=0$ for $1\le i\le m$ and
$\epsilon_i=1$ for $m<i\le m+n$.
We assume that $\pi$ has degree $(0,-2)$ and choose
its coefficient expression so that
$C^{i_1,i_2}_{j_1,j_2}=0$ unless
$\epsilon_{i_1}+\epsilon_{i_2}
=\epsilon_{j_1}+\epsilon_{j_2}$.
All indices below range from $1$ to $m+n$.
Under this correspondence, a quadratic bivector field  
\begin{align}\label{pi1}
    \pi = \sum_{i_1,i_2,j_1,j_2} C^{i_1,i_2}_{j_1,j_2} z_{i_1}z_{i_2}\partial_{z_{j_1}}\wedge \partial_{z_{j_2}}
\end{align}
on $\mathcal{O}(k^{m||n})$ is mapped to the quadratic bivector field  
\begin{align}\label{pi2}
    \pi^! = \sum_{i_1,i_2,j_1,j_2} C^{i_1,i_2}_{j_1,j_2}   \partial_{\zeta_{i_1}}\wedge \partial_{\zeta_{i_2}}\zeta_{j_1}\zeta_{j_2}
\end{align}
on $\mathcal{O}(k^{n\wedge m})$, and vice versa.

\begin{theorem}[Chen--Yu \cite{CY}]\label{quadratic poisson Koszul}
Let $\pi$ and $\pi^!$ be the quadratic bivectors on $k^{m||n}$ and $k^{n\wedge m}$ related by \eqref{corres}. Then $\pi$ is Poisson if and only if $\pi^!$ is Poisson.
\end{theorem}

We next describe this correspondence for diagonal
Poisson structures.

\begin{proposition}
\label{prop:diagonal-koszul}
A Poisson structure on $\mathcal{O}(k^{m||n})$
is diagonal with respect to $x_1,\ldots,x_m$, $y_1,\ldots,y_n$
if and only if it has the form
\begin{equation}\label{eq:diagonal-original}
\begin{aligned}
\pi=
\sum_{1\leq i<j\leq m}
a_{ij}x_i x_j\,
\partial_{x_i}\wedge\partial_{x_j}
+
\sum_{1\leq s<t\leq n}
b_{st}y_s y_t\,
\partial_{y_s}\wedge\partial_{y_t}
+
\sum_{i=1}^{m}\sum_{s=1}^{n}
c_{is}x_i y_s\,
\partial_{x_i}\wedge\partial_{y_s},
\end{aligned}
\end{equation}
where $a_{ij},b_{st},c_{is}\in k$.
Under \eqref{corres}, its Koszul dual is
\begin{equation}\label{eq:diagonal-dual}
\begin{aligned}
\pi^!=
\sum_{1\leq i<j\leq m}
a_{ij}\xi_i\xi_j\,
\partial_{\xi_i}\wedge\partial_{\xi_j}
+
\sum_{1\leq s<t\leq n}
b_{st}\eta_s\eta_t\,
\partial_{\eta_s}\wedge\partial_{\eta_t}
-
\sum_{i=1}^{m}\sum_{s=1}^{n}
c_{is}\xi_i\eta_s\,
\partial_{\xi_i}\wedge\partial_{\eta_s}.
\end{aligned}
\end{equation}
In particular, $\pi^!$ is diagonal with respect to
$\xi_1,\ldots,\xi_m,\eta_1,\ldots,\eta_n$.
Every choice of $a_{ij},b_{st},c_{is}\in k$ gives a Poisson structure. Conversely, every Poisson structure diagonal in the specified dual coordinates is obtained in this way.
\end{proposition}

\begin{proof}
The first expression follows from
Definition~\ref{def:diagonalPoisson}, with the coordinates
ordered as $x_1,\ldots,x_m,y_1,\ldots,y_n$.
Let $E_i=x_i\frac{\partial}{\partial x_i}$ and $F_s=y_s\frac{\partial}{\partial y_s}$ be the coordinate Euler derivations, each of bidegree $(0,0)$. They commute pairwise, and \eqref{eq:diagonal-original} is a constant linear combination of their pairwise polyvector products. The Schouten Leibniz rule therefore gives $[\pi,\pi]=0$.

Let $\Psi$ denote the polyvector correspondence
\eqref{corres}.
Reordering its images according to the bigraded sign rule gives
$\Psi\bigl(x_i x_j\,
\partial_{x_i}\wedge\partial_{x_j}\bigr)
=\xi_i\xi_j\,
\partial_{\xi_i}\wedge\partial_{\xi_j}$, 
$\Psi\bigl(y_s y_t\,
\partial_{y_s}\wedge\partial_{y_t}\bigr)
=\eta_s\eta_t\,
\partial_{\eta_s}\wedge\partial_{\eta_t}$ and 
$\Psi\bigl(x_i y_s\,
\partial_{x_i}\wedge\partial_{y_s}\bigr)
=-\xi_i\eta_s\,
\partial_{\xi_i}\wedge\partial_{\eta_s}.$
This proves \eqref{eq:diagonal-dual}.
Theorem~\ref{quadratic poisson Koszul} implies that
$\pi^!$ is Poisson.
Finally, the coefficient transformation $(a_{ij},b_{st},c_{is})\mapsto(a_{ij},b_{st},-c_{is})$ is invertible, so every structure diagonal in the specified dual coordinates arises in this way.
\end{proof}

We next show that the correspondence identifies the modular vector fields and preserves diagonalizability, as needed for the twisted BV constructions.
We equip $A$ and $A^!$ with the mixed volume forms
$$
\gamma
=
dx_1\wedge\cdots\wedge dx_m
\wedge y_1^*\wedge\cdots\wedge y_n^*,
\qquad
\gamma^!
=
\xi_1^*\wedge\cdots\wedge\xi_m^*
\wedge d\eta_1\wedge\cdots\wedge d\eta_n,
$$
respectively. The modular vector fields are defined with
respect to these forms.
\begin{proposition}\label{modulardual}
Let $A=\mathcal O(k^{m||n})$ and
$A^!=\mathcal O(k^{n\wedge m})$ be equipped with
Koszul-dual quadratic Poisson structures $\pi$ and $\pi^!$
of degree $(0,-2)$, respectively.
Let $\upsilon$ and $\upsilon^!$ be their modular vector
fields with respect to the mixed volume forms
$\gamma$ and $\gamma^!$ chosen above.
Then the correspondence \eqref{corres} sends
$\upsilon$ to $\upsilon^!$.
Moreover, the derivation of $A$ corresponding to $\upsilon$ is diagonalizable over $k$ if and only if the derivation of $A^!$ corresponding to $\upsilon^!$ is.
\end{proposition}

\begin{proof}
Let $\Psi$ denote the isomorphism of polyvector algebras
given by \eqref{corres}, and let $\Phi$ be the corresponding
isomorphism of mixed differential forms, determined by
$$x_i\longmapsto\partial_{\xi_i},\quad
dx_i\longmapsto\xi_i^*,\quad
\partial_{y_j}\longmapsto\eta_j,\quad
y_j^*\longmapsto d\eta_j.$$
By \cite[Proposition~3.19]{CY}, these maps satisfy
$\Phi\tilde d=\tilde d\Phi$ and $
\Phi\tilde\iota_P=\tilde\iota_{\Psi(P)}\Phi.$
Moreover, the chosen volume forms satisfy
$\Phi(\gamma)=\gamma^!$.
Writing $D_\gamma(P)=\tilde\iota_P\gamma$ and
$D_{\gamma^!}(Q)=\tilde\iota_Q\gamma^!$, we therefore obtain
$\Phi D_\gamma=D_{\gamma^!}\Psi$.
It follows that
$$\Psi\mathrm{Div}_{\gamma}
=\Psi D_\gamma^{-1}\tilde dD_\gamma
=D_{\gamma^!}^{-1}\Phi\tilde dD_\gamma
=D_{\gamma^!}^{-1}\tilde dD_{\gamma^!}\Psi
=\mathrm{Div}_{\gamma^!}\Psi.$$
Since $\Psi(\pi)=\pi^!$, this gives
$\Psi(\upsilon)
=-\Psi\mathrm{Div}_{\gamma}(\pi)
=-\mathrm{Div}_{\gamma^!}(\pi^!)
=\upsilon^!.$

Let $M_x$ and $M_y$ denote the matrices of
$\overline\upsilon$ on
$\operatorname{span}_k\{x_1,\ldots,x_m\}$ and
$\operatorname{span}_k\{y_1,\ldots,y_n\}$, respectively.
On the spaces of linear generators, the modular derivations of $A$ and $A^!$ have matrices
$$
\begin{pmatrix} M_x&0\\0&M_y\end{pmatrix}
\qquad\text{and}\qquad
\begin{pmatrix} M_x^{\mathsf T}&0\\0&-M_y^{\mathsf T}\end{pmatrix},
$$
respectively. These matrices are diagonalizable over $k$ simultaneously.

For either algebra, diagonalizability on the space of
linear generators is equivalent to diagonalizability
on the whole algebra. Indeed, a homogeneous eigenbasis
of generators yields a monomial eigenbasis by the
Leibniz rule. Conversely, the restriction of a
diagonalizable operator to an invariant subspace is
diagonalizable.
Hence the derivation of $A$ associated with $\upsilon$ is
diagonalizable if and only if the derivation of
$A^!$ associated with $\upsilon^!$ is.
\end{proof}

\begin{remark}
For the coefficient expressions \eqref{pi1} and
\eqref{pi2}, the modular vector fields are given by
$$
\begin{aligned}
\upsilon
={}&
\sum_{i_1,i_2,j_2}
(-1)^{1+\epsilon_{i_1}
+\epsilon_{i_1}\epsilon_{i_2}}
C^{i_1,i_2}_{i_1,j_2}\,
z_{i_2}\partial_{z_{j_2}}
+
\sum_{i_1,i_2,j_1}
(-1)^{\epsilon_{i_1}}
C^{i_1,i_2}_{j_1,i_1}\,
z_{i_2}\partial_{z_{j_1}}
\\
&+
\sum_{i_1,i_2,j_2}
(-1)^{1+\epsilon_{i_2}}
C^{i_1,i_2}_{i_2,j_2}\,
z_{i_1}\partial_{z_{j_2}}
+
\sum_{i_1,i_2,j_1}
(-1)^{\epsilon_{i_2}
+\epsilon_{i_1}\epsilon_{i_2}}
C^{i_1,i_2}_{j_1,i_2}\,
z_{i_1}\partial_{z_{j_1}},
\end{aligned}
$$
and
$$
\begin{aligned}
\upsilon^!
={}&
\sum_{i_1,i_2,j_2}
(-1)^{1+\epsilon_{i_1}+\epsilon_{i_2}
+\epsilon_{i_1}\epsilon_{i_2}}
C^{i_1,j_2}_{i_1,i_2}\,
\zeta_{i_2}\partial_{\zeta_{j_2}}
+
\sum_{i_1,i_2,j_1}
(-1)^{1+\epsilon_{i_1}+\epsilon_{i_2}}
C^{j_1,i_1}_{i_1,i_2}\,
\zeta_{i_2}\partial_{\zeta_{j_1}}
\\
&+
\sum_{i_1,i_2,j_2}
(-1)^{\epsilon_{i_1}+\epsilon_{i_2}}
C^{i_2,j_2}_{i_1,i_2}\,
\zeta_{i_1}\partial_{\zeta_{j_2}}
+
\sum_{i_1,i_2,j_1}
(-1)^{\epsilon_{i_1}+\epsilon_{i_2}
+\epsilon_{i_1}\epsilon_{i_2}}
C^{j_1,i_2}_{i_1,i_2}\,
\zeta_{i_1}\partial_{\zeta_{j_1}}.
\end{aligned}
$$
Here all indices range from $1$ to $m+n$.
The signs in the correspondence are determined by
$$
\Psi(z_i\partial_{z_j})
=
\partial_{\zeta_i}\zeta_j
=
(-1)^{\epsilon_i\epsilon_j}
\zeta_j\partial_{\zeta_i}.
$$
\end{remark}

\begin{corollary}
Let $A=\mathcal{O}(k^{m||n})$ or
$\mathcal{O}(k^{n\wedge m})$, equipped with the corresponding
mixed volume form chosen above.
If $\pi$ is diagonal with respect to the specified
homogeneous linear coordinates, then its modular vector
field $\upsilon=-\mathrm{Div}_{\gamma}(\pi)$ is
diagonalizable over $k$.
\end{corollary}

\begin{proof}
For $A=\mathcal{O}(k^{m||n})$, the divergence of a diagonal
bivector is a linear combination of the coordinate Euler
vector fields. Thus the ordinary derivation corresponding
to $\upsilon$ acts diagonally on the monomial basis of $A$.
The dual case follows from
Proposition~\ref{modulardual} and the preservation of
diagonal Poisson structures under \eqref{corres}.
\end{proof}

\subsection{Batalin--Vilkovisky structures under Koszul duality}

The modular compatibility above identifies the twisted chain complexes under the Koszul correspondence. We now compare the Poincar\'e duality maps and prove that the calculus isomorphism intertwines the operators $B$ obtained from the zero-eigenvalue subcomplexes. It follows that the BV operators correspond in the non-unimodular case as well. 
For the unimodular quadratic polynomial case, see also
\cite[Theorem~1.3]{chen2021poisson}.

\begin{theorem}\label{comdiagram}
    Let $A=(\mathcal{O}(k^{m||n}),\pi)$ be a quadratic Poisson algebra with $|\pi|=(0,-2)$, and let $A^!=(\mathcal{O}(k^{n\wedge m}),\pi^!)$ be its Koszul dual Poisson algebra.
    Their modular vector fields are denoted by $\upsilon$ and $\upsilon^!$, respectively.
    Then the following diagram commutes:
\begin{equation}\label{diagPoisson}
\xymatrixcolsep{4pc}
\xymatrix{
\mathrm{HP}^{\bullet,\bullet}(A)
\ar[r]^-{\cong}
\ar[d]_{\cong}
&
\widetilde{\mathrm{HP}}_{-n-\bullet,m-\bullet}
(A;\upsilon)
\ar[d]_{\cong}
\\
\mathrm{HP}^{\bullet,\bullet}(A^!)
\ar[r]^-{\cong}
&
\widetilde{\mathrm{HP}}_{-n-\bullet,m-\bullet}
(A^!;\upsilon^!).
}
\end{equation}
\end{theorem}
\begin{proof}
Recall the polyvector and mixed form isomorphisms $\Psi$ and $\Phi$ from the proof of Proposition~\ref{modulardual}:
\begin{align}\label{correspondence}
\Psi:&\quad x_i\mapsto\partial_{\xi_i},\quad
 y_j\mapsto\partial_{\eta_j},\quad
 \partial_{x_i}\mapsto\xi_i,\quad\partial_{y_j}\mapsto\eta_j,\\
\Phi:&\quad x_i\mapsto\partial_{\xi_i},\quad
 y_j^*\mapsto d\eta_j,\quad dx_i\mapsto\xi_i^*,\quad
 \partial_{y_j}\mapsto\eta_j.\label{correspondence2}
\end{align}
They preserve degrees. The identities
$\Psi[\pi,P]=[\pi^!,\Psi(P)]$ and
$\Phi\tilde\iota_P=\tilde\iota_{\Psi(P)}\Phi$, together
with $\Phi\tilde d=\tilde d\Phi$ and
$\Psi(\upsilon)=\upsilon^!$, give
$$
\Psi\delta^\pi=\delta^{\pi^!}\Psi,\qquad
\Phi\partial_\upsilon=\partial_{\upsilon^!}\Phi.
$$
Thus $\Psi$ and $\Phi$ induce the left and right vertical
isomorphisms in \eqref{diagPoisson}, respectively.
The horizontal maps are induced by volume contraction,
as in Theorem~\ref{horiso}. Finally, the identity
$$
\Phi D_\gamma=D_{\gamma^!}\Psi,
$$
proved in Proposition~\ref{modulardual}, shows that the square commutes on representatives and hence on (co)homology. This argument applies without a diagonalizability hypothesis.
\end{proof}

Theorem~\ref{comdiagram} establishes compatibility with volume duality.
To obtain a BV isomorphism, it remains to compare
the operator $B$ on twisted Poisson homology. The next theorem
shows that the calculus correspondence respects $B$ and the
other calculus operations.
\begin{theorem}\label{isodiff}
Under the hypotheses of Theorem~\ref{comdiagram},
assume that the modular derivation of $A$ is
diagonalizable over $k$.
Then $(\Psi_*,\Phi_*)$ is an isomorphism between
the differential calculi with duality
$$
\bigl(
\mathrm{HP}^{\bullet,\bullet}(A),
\widetilde{\mathrm{HP}}_{\bullet,\bullet}(A;\upsilon),
\wedge,\tilde\iota,[-,-],\tilde d,[\gamma]
\bigr),
\bigl(
\mathrm{HP}^{\bullet,\bullet}(A^!),
\widetilde{\mathrm{HP}}_{\bullet,\bullet}
(A^!;\upsilon^!),
\wedge,\tilde\iota,[-,-],\tilde d,[\gamma^!]
\bigr).
$$
\end{theorem}

\begin{proof}
By Proposition~\ref{modulardual}, the modular
derivation of $A^!$ is also diagonalizable.
Corollary~\ref{diffcaldual} therefore supplies both
differential calculi with duality.

By \cite[Proposition~3.19]{CY}, the maps $\Psi$ and
$\Phi$ satisfy
$\Psi(P\wedge Q)=\Psi(P)\wedge\Psi(Q)$,
$\Psi[P,Q]=[\Psi(P),\Psi(Q)]$,
and
$\Phi\tilde\iota_P=\tilde\iota_{\Psi(P)}\Phi$,
$\Phi\tilde d=\tilde d\Phi$.
Theorem~\ref{comdiagram} shows that these maps
intertwine the Poisson differentials. Hence the
product, bracket and contraction identities descend
to (co)homology.

Moreover, $\Psi(\upsilon)=\upsilon^!$ gives
$\Phi L_\upsilon=L_{\upsilon^!}\Phi$.
Thus $\Phi$ preserves eigenspaces and intertwines the
zero-eigenvalue inclusions and projections.
Consequently,
$\Phi_*\tilde d=\tilde d\Phi_*$,
where $\tilde d$ on twisted Poisson homology denotes
the transported operator defined after
Theorem~\ref{diagonalizable}.

Finally, $\Phi(\gamma)=\gamma^!$, so $\Phi_*$
preserves the fundamental classes. This proves
compatibility with all the calculus operations
and with the duality.
\end{proof}

\begin{proof}[Proof of Theorem~\ref{mainthm}]
By Proposition~\ref{modulardual}, the dual modular derivation is diagonalizable. Theorem~\ref{mainthm1.1} therefore gives BV structures on both Poisson cohomologies. By Theorems~\ref{comdiagram} and~\ref{isodiff}, the induced maps satisfy $\Phi_*\mathrm{PD}=\mathrm{PD}^!\Psi_*$ and $\Phi_*B=B^!\Phi_*$. Hence
$$
\Psi_*\Delta
=-\Psi_*\mathrm{PD}^{-1}B\mathrm{PD}
=-(\mathrm{PD}^!)^{-1}B^!\mathrm{PD}^!\Psi_*
=\Delta^!\Psi_*.
$$
The map $\Psi_*$ also preserves the product and bracket,
and hence is an isomorphism of BV algebras.
\end{proof}

An analogous compatibility result for Hochschild BV structures is given in \cite[Theorem~1.1]{liu2019koszul}.

\section{Examples}\label{sect:examples}

We begin with non-unimodular diagonal quadratic Poisson structures on $k^{2||2}$ and $k^{1||3}$. We compute their diagonalizable modular derivations and the BV operators on coordinate Euler classes, then compare the operators under Koszul duality. The final two examples consider a non-diagonal structure with diagonalizable modular derivation and a structure for which diagonalizability fails.

By Proposition~\ref{prop:canonical-volume}, a Poisson structure is non-unimodular if and only if its modular vector field is nonzero. In the first two examples, we also verify directly that the modular class is nonzero.

\begin{example}[A non-unimodular structure on $k^{2||2}$]\label{ex:nonunimodular-22}
Consider $A=\mathcal O(k^{2||2})$, with coordinates
$x_1,x_2,y_1,y_2$, and let
$$
\begin{aligned}
\pi=
x_1x_2\partial_{x_1}\wedge\partial_{x_2}
-y_1y_2\partial_{y_1}\wedge\partial_{y_2}
+\sum_{i,j=1}^{2}
x_i y_j\partial_{x_i}\wedge\partial_{y_j}.
\end{aligned}
$$
The bivector $\pi$ is diagonal and quadratic of degree $(0,-2)$. It is a linear combination of polyvector products of coordinate Euler derivations, which commute pairwise. Hence $[\pi,\pi]=0$.

For the mixed volume form
$\gamma=dx_1\wedge dx_2\wedge y_1^*\wedge y_2^*$,
direct mixed contraction gives
$$
\upsilon=-\mathrm{Div}_{\gamma}(\pi)
=-x_1\partial_{x_1}
-3x_2\partial_{x_2}
-3y_1\partial_{y_1}
-y_2\partial_{y_2}.
$$
In particular,
$\upsilon
(x_1^a x_2^b y_1^{\varepsilon_1}y_2^{\varepsilon_2})=
-(a+3b+3\varepsilon_1+\varepsilon_2)
x_1^a x_2^b y_1^{\varepsilon_1}y_2^{\varepsilon_2}$,
where $\varepsilon_1,\varepsilon_2\in\{0,1\}$.
Thus the modular derivation is nonzero and diagonalizable. The structure $\pi$ is therefore non-unimodular, and Theorem~\ref{mainthm1.1} gives a BV structure on $\mathrm{HP}^{\bullet,\bullet}(A)$.

The coordinate Euler fields commute with $\pi$ and $\upsilon$.
Since
$\mathrm{Div}_\gamma(x_i\partial_{x_i})=1$ and
$\mathrm{Div}_\gamma(y_j\partial_{y_j})=-1$,
$\Delta[x_i\partial_{x_i}]=-[1]$,
$\Delta[y_j\partial_{y_j}]=[1]$.
Since the Euler fields commute, \eqref{eq:second-order} shows that $\Delta$ restricts to a graded derivation on the subalgebra generated by their classes.
In particular,
\begin{align*}
\Delta[x_1x_2\partial_{x_1}\wedge\partial_{x_2}]
&=[x_1\partial_{x_1}]-[x_2\partial_{x_2}],\\
\Delta[y_1y_2\partial_{y_1}\wedge\partial_{y_2}]
&=[y_1\partial_{y_1}]-[y_2\partial_{y_2}],\\
\Delta[x_i y_j\partial_{x_i}\wedge\partial_{y_j}]
&=-[x_i\partial_{x_i}]-[y_j\partial_{y_j}].
\end{align*}
To see that $[\upsilon]\ne0$, assign polynomial degree one
to each coordinate. Since $\pi$ has quadratic coefficients,
the coefficients of $[f,\pi]$ have no terms of degree less
than two for any $f\in A$. 
Thus $\upsilon$, a nonzero vector field with linear
coefficients, is not of the form $[f,\pi]$ for any $f\in A$.
Since the Poisson differential increases polyvector length
by one, this excludes all possible coboundaries for $\upsilon$.
By Remark~\ref{rem:modular-class-bv}, we therefore obtain
$\Delta[\pi]=[\upsilon]\ne0$.

On $A^!=\mathcal O(k^{2\wedge2})$, the dual Poisson bivector  is
$$
\pi^!=\xi_1\xi_2\partial_{\xi_1}\wedge\partial_{\xi_2}
-\eta_1\eta_2\partial_{\eta_1}\wedge\partial_{\eta_2}
-\sum_{i,j=1}^{2}\xi_i\eta_j
\partial_{\xi_i}\wedge\partial_{\eta_j}.
$$
For $\gamma^!=\xi_1^*\wedge\xi_2^*\wedge d\eta_1\wedge d\eta_2$,
$\upsilon^!=-\xi_1\partial_{\xi_1}-3\xi_2\partial_{\xi_2}
+3\eta_1\partial_{\eta_1}+\eta_2\partial_{\eta_2}$.
The dual modular derivation is again nonzero and diagonalizable.
Under the correspondence
$x_i\partial_{x_i}\mapsto\xi_i\partial_{\xi_i}$,
$y_j\partial_{y_j}\mapsto-\eta_j\partial_{\eta_j}$,
Theorem~\ref{mainthm} identifies the BV operators.
\end{example}

\begin{example}[A non-unimodular structure on $k^{1||3}$]\label{ex:nonunimodular-13}
Consider $A=\mathcal O(k^{1||3})$, with coordinates
$x,y_1,y_2,y_3$ of degrees
$|x|=(0,0)$ and $|y_j|=(1,0)$.
Let
$$
\pi=
\sum_{j=1}^{3}xy_j\partial_x\wedge\partial_{y_j}
-\sum_{1\leq i<j\leq3}
y_i y_j\partial_{y_i}\wedge\partial_{y_j}.
$$
As in the preceding example, this is a diagonal quadratic
Poisson structure of degree $(0,-2)$.

For
$\gamma=dx\wedge y_1^*\wedge y_2^*\wedge y_3^*$,
the modular vector field is
$$
\upsilon=-\mathrm{Div}_{\gamma}(\pi)
=-3x\partial_x
-3y_1\partial_{y_1}
-y_2\partial_{y_2}
+y_3\partial_{y_3}.
$$
Its action on monomials is
$\upsilon
(x^a y_1^{\varepsilon_1}
y_2^{\varepsilon_2}y_3^{\varepsilon_3})
=
(-3a-3\varepsilon_1-\varepsilon_2+\varepsilon_3)
x^a y_1^{\varepsilon_1}
y_2^{\varepsilon_2}y_3^{\varepsilon_3},$
where $\varepsilon_j\in\{0,1\}$.
Hence the modular derivation is diagonalizable. Since $\upsilon\ne0$, Proposition~\ref{prop:canonical-volume} excludes a mixed volume form that is a Poisson cycle, so $\pi$ is non-unimodular.
Its modular class is nonzero by the polynomial degree
argument in Example~\ref{ex:nonunimodular-22}.
Thus Theorem~\ref{mainthm1.1} applies here as well.

On the subalgebra generated by the coordinate Euler classes,
$\Delta$ is the graded derivation determined by
$\Delta[x\partial_x]=-[1]$,
$\Delta[y_j\partial_{y_j}]=[1]$ for $j=1,2,3$.
Consequently,
\begin{align*}
\Delta[xy_j\partial_x\wedge\partial_{y_j}]
=-[x\partial_x]-[y_j\partial_{y_j}],\quad
\Delta[y_i y_j\partial_{y_i}\wedge\partial_{y_j}]
=[y_i\partial_{y_i}]-[y_j\partial_{y_j}]
\quad(i<j).
\end{align*}
Thus
$\Delta[\pi]=-3[x\partial_x]-3[y_1\partial_{y_1}]
-[y_2\partial_{y_2}]+[y_3\partial_{y_3}]
=[\upsilon]\ne0$.

On $A^!=\mathcal O(k^{3\wedge1})$, the dual structure is
$$
\pi^!=-\sum_{j=1}^{3}\xi\eta_j
\partial_\xi\wedge\partial_{\eta_j}
-\sum_{1\leq i<j\leq3}\eta_i\eta_j
\partial_{\eta_i}\wedge\partial_{\eta_j}.
$$
For $\gamma^!=\xi^*\wedge d\eta_1\wedge d\eta_2\wedge d\eta_3$,
$\upsilon^!=-3\xi\partial_\xi
+3\eta_1\partial_{\eta_1}
+\eta_2\partial_{\eta_2}-\eta_3\partial_{\eta_3}$.
The resulting modular derivation of $A^!$ is nonzero and diagonalizable.
The correspondence
$x\partial_x\mapsto\xi\partial_\xi$,
$y_j\partial_{y_j}\mapsto-\eta_j\partial_{\eta_j}$
identifies the BV operators by Theorem~\ref{mainthm}.
\end{example}

In each of the two examples, the modular vector field is nonzero for every mixed volume form, while its diagonalizability permits the BV construction. Koszul duality identifies the resulting BV algebra with that of the corresponding dual Poisson algebra, extending the unimodular comparison of \cite{CY} as in Theorem~\ref{mainthm}.

The next example has a diagonalizable modular derivation although its Poisson structure is not diagonal in any homogeneous linear coordinates.

\begin{example}[A non-diagonal structure]\label{ex:nondiagonal}
On $k^{3||0}$, consider the Poisson structure
$$
\pi=x_1^2\partial_{x_2}\wedge\partial_{x_3}
+x_1x_2\partial_{x_1}\wedge\partial_{x_2}
-x_1x_3\partial_{x_1}\wedge\partial_{x_3}.
$$
For the standard volume,
$$
\upsilon=-x_2\partial_{x_2}+x_3\partial_{x_3},\qquad
\upsilon(x_1^a x_2^b x_3^c)=(c-b)x_1^a x_2^b x_3^c.
$$
Hence the modular derivation is diagonalizable and nonzero.
The zero locus of $\pi$ is the hyperplane $x_1=0$.
In contrast, a diagonal bivector vanishes on every
coordinate axis, so its zero locus spans the whole
ambient space. The dimension of the linear span of the zero locus is invariant under invertible linear coordinate changes. Hence $\pi$ cannot be diagonal in any linear coordinates.
Theorem~\ref{mainthm} applies to this non-unimodular
structure and its Koszul dual.
\end{example}

\begin{example}[Failure of diagonalizability]\label{ex:nilpotent-modular}
On $k^{2||0}$, let $\pi=x_1^2\partial_{x_1}\wedge\partial_{x_2}$.
It is Poisson because the space has dimension two.
For $\gamma=dx_1\wedge dx_2$, our contraction order gives
$D_\gamma(\pi)=-x_1^2$ and $\upsilon=-2x_1\partial_{x_2}$.
On $\operatorname{span}\{x_1,x_2\}$, the modular derivation is a nonzero nilpotent operator. Hence the modular derivation is not diagonalizable. Proposition~\ref{modulardual} implies the same conclusion for the dual modular derivation.
\end{example}

\begin{remark}\label{rem:homogeneous-scope}
On $\mathcal O(k^{m||n})$, suppose that the coefficients
of $\pi$ are homogeneous of polynomial degree $d$.
Then its modular derivation shifts polynomial degree by
$d-2$. If $d>2$, the lowest nonzero homogeneous term of
an eigenvector rules out a nonzero eigenvalue; if $d<2$,
the highest nonzero term does so. Thus, when $d\ne2$,
all eigenvalues are zero, and diagonalizability forces
$\upsilon=0$. 
Among Poisson structures homogeneous in polynomial degree, the non-unimodular case of Theorem~\ref{mainthm1.1} is therefore quadratic. The theorem also applies to inhomogeneous Poisson structures of degree $(0,-2)$.
\end{remark}

\begin{ack}
We are grateful to Xiaojun Chen for his guidance and to Zhuyang Li for the discussions and exchanges throughout this work.
This work was supported by the Zhejiang Provincial Natural Science Foundation (Grant No. LQ24A010020) and the  National Natural Science Foundation of China (Grant Nos. 12271377 and 12261131498).
\end{ack}

\end{document}